\documentclass{article}
\usepackage{amsmath}
\usepackage{amssymb}
\usepackage{amsfonts}
\usepackage{amsthm}
\usepackage{tikz-cd}
\usepackage{rotating}
\usepackage[colorlinks=true,linkcolor=blue,citecolor=blue]{hyperref}
\usepackage{url}

\numberwithin{equation}{section}

\newtheoremstyle{sltheorem}%
  {5pt}
  {1pt}
  {\slshape}
  {}
  {\bfseries}
  {.}
  {.5em}
  {}

\theoremstyle{sltheorem}
  \newtheorem{THM}{Theorem}[section]
  \newtheorem{LEM}[THM]{Lemma}
  
  \newtheorem{COR}[THM]{Corollary}

\theoremstyle{definition}
  \newtheorem{DEF}[THM]{Definition}
  \newtheorem{EX}[THM]{Example}

\newcommand{\Emb}{\mathrm{Emb}}
\newcommand{\Ob}{\mathrm{Ob}}
\newcommand{\ord}{\mathrm{ord}}
\newcommand{\id}{\mathrm{id}}
\newcommand{\Fraisse}{Fra\"\i ss\'e}

\let\le\leqslant
\let\ge\geqslant

\let\geq\geqslant

\newcommand{\BB}{\mathbf{B}}
\newcommand{\CC}{\mathbf{C}}

\newcommand{\0}{\emptyset}

\newcommand{\Aut}{\mathrm{Aut}}
\newcommand{\im}{\mathrm{im}}

\newcommand{\frakg}{\mathfrak{g}}

\newcommand{\calD}{\mathcal{D}}
\newcommand{\calF}{\mathcal{F}}
\newcommand{\calP}{\mathcal{P}}

\newcommand{\NN}{\mathbb{N}}
\newcommand{\ZZ}{\mathbb{Z}}

\let\phi\varphi

\title{Hall's universal group does not have finite big Ramsey degrees}
\author{Dragan Ma\v sulovi\'c\\
Department of Mathematics and Informatics\\
Faculty of Sciences, University of Novi Sad, Serbia\\
email: \texttt{dragan.masulovic@dmi.uns.ac.rs}\\
\and
Veljko Tolji\'c\\
Freie Universit\"at Berlin, Germany\\
email: \texttt{veljko.toljic@fu-berlin.de}
}

\date{\today}

\begin{document}

\maketitle

\begin{center}
  \it Dedicated to Jaroslav Ne\v set\v ril\\
  \it on the occasion of his 80th birthday
\end{center}

\begin{abstract}
  In this paper we show that the Hall's universal group does not have finite big Ramsey degrees.
  Our strategy consists of piggybacking on the recent result of
  Hubi\v cka, Kone\v cn\'y, Todor\v cevi\'c and Zucker (announced at EUROCOMB~2025)
  that the random edge-labelled graph (the \Fraisse\ limit of the class of all finite complete edge-labelled graphs
  where the set of labels is countably infinite) does not have finite big Ramsey degrees.
  We then use the machinery of category theory to transport this result from the context of
  edge-labelled graphs to the context of groups. The main step in the process is the construction of a
  functor from the category of edge-labelled graphs and embeddings to the category of groups
  and group embeddings which takes finite graphs to finite groups. This makes is possible for us
  to build a subgroup of the Hall's universal group which encodes the random edge-labelled graph.
  Our main result shows that several families finite groups, including non-cyclic abelian groups,
  do not have finite big Ramsey degrees in the Hall's universal group.
  Actually, we show that all potential candidates for groups with finite big Ramsey degrees have to be solvable.
  
\end{abstract}

\section{Introduction}

In this paper we show that the Hall's universal group does not have finite big Ramsey degrees.
Our strategy consists of piggybacking on the recent result of
Hubi\v cka, Kone\v cn\'y, Todor\v cevi\'c and Zucker~\cite{hubicka-omega-e-lab-hypergr}
that the random edge-labelled graph (the \Fraisse\ limit of the class of all finite complete edge-labelled graphs
where the set of labels is countably infinite) does not have finite big Ramsey degrees.
We then use the machinery of category theory to transport this result from the context of
edge-labelled graphs to the context of groups.

In Section~\ref{hall.sec.prelim} we review some basic facts about the Hall's universal group,
big Ramsey degrees, and outline several convenient categorical tools to transport the
results about big Ramsey degrees from one category to the other.
In Section~\ref{hall.sec.ftr} we construct a functor from the category of edge-labelled graphs
and embeddings to the category of groups and group embeddings which takes finite edge-labelled-graphs to finite groups.
This makes is possible for us to build a subgroup of the Hall's universal group which encodes the random edge-labelled graph.
As a spin-off, we get a very natural functor from the category of simple graphs and graph embeddings into
the category of groups and group embeddings that takes finite graphs to finite groups (cf.~\cite{HeQiao} where
one can find an overview of attempts to construct such a functor for the category of graphs and graph homomorphisms).
In Section~\ref{hall.sec.brd} we prove that several families finite groups, including non-cyclic abelian groups,
do not have finite big Ramsey degrees in the Hall's universal group.
Actually, we show that all potential candidates for groups with finite big Ramsey degrees have to be solvable.
All the group theoretic results that we need in the paper are concentrated in Appendix~\ref{hall.sec.A}.

\section{Preliminaries}
\label{hall.sec.prelim}

\paragraph{Hall's universal group.}
Let $H$ denote the \emph{Hall's universal group}, the \Fraisse\ limit
of the class of all finite groups. 
Hall's universal group is locally finite, and embeds every finite or
countably infinite locally finite group. Moreover, $H$ is ultrahomogeneous:
every isomorphism between two finite subgroups of $H$ extends to an automorphism of~$H$.
Consequently, $H$ has the \emph{extension property}: for every pair of finite groups $A$ and $B$, every embedding
$f : A \hookrightarrow B$ and every embedding $g : A \hookrightarrow H$ there is an embedding
$g' : B \hookrightarrow H$ such that $g' \circ f = g$:
\begin{center}
  \begin{tikzcd}
    B \arrow[dr, hookrightarrow, dashed, "g'"] & \\
    A \arrow[u, hookrightarrow, "f"] \arrow[r, hookrightarrow, "g"'] & H
  \end{tikzcd}
\end{center}

By $\langle a_1, \ldots, a_n\rangle_G$ we denote the subgroup of $G$ generated by 
$a_1, \ldots, a_n \in G$ and by $\ord_G(a)$ we denote the order of $a \in G$.

If $G, G_1, \ldots, G_n$ are groups and $h_i : G \to G_i$, $1 \le i \le n$, are homomorphisms,
then by $[h_1, \ldots, h_n] : G \to G_1 \times \ldots \times G_n$ we denote the cotupling of
$h_1, \ldots, h_n$ defined by $[h_1, \ldots, h_n](g) = (h_1(g), \ldots, h_n(g))$.

\paragraph{Big Ramsey degrees.}
Big Ramsey degrees were first introduced in the context of structural Ramsey theory in~\cite{KPT}.
For first-order structures $A$ and $B$ with the same signature, let $\Emb(A, B)$ denote
the set of all embeddings $A \hookrightarrow B$ and let $\binom BA$ denote the set of all substructures of $B$
isomorphic to~$A$.

Let $A$ and $S$ be first-order structures such that $A$ is finite and $S$ is locally finite.
An \emph{embedding big Ramsey degree of $A$ in $S$} is the least positive integer $n \in \NN$, if such an integer exists,
with the property that for every $k \in \NN$ and every coloring
$\chi : \Emb(A, S) \to k$ one can find a $w \in \Emb(S, S)$ satisfying $|\chi(w \circ \Emb(A, S))| \le n$.
We then write $T(A, S) = n$. If no such $n \in \NN$ exists we write $T(A, S) = \infty$.
We say that \emph{$S$ has finite embedding big Ramsey degrees} if $T(A, S) < \infty$
for every finitely generated substructure $A$ of~$S$.

A \emph{structural big Ramsey degree of $A$ in $S$} is the least positive integer $n \in \NN$, if such an integer exists,
with the property that for every $k \in \NN$ and every coloring
$\chi : \binom SA \to k$ one can find a substructure $S' \le S$ such that $S' \cong S$ and
$|\{\chi(A') : A' \in \binom{S'}{A}| \le n$.
We then write $\widetilde T(A, S) = n$. If no such $n \in \NN$ exists we write $\widetilde T(A, S) = \infty$.
We say that \emph{$S$ has finite structural big Ramsey degrees} if $\widetilde T(A, S) < \infty$
for every finitely generated substructure $A$ of~$S$.

The relationship between embedding and structural big Ramsey degrees is straightforward, as demonstrated by Zucker~\cite{zucker1}:
if $\Aut(A)$ is finite then
\begin{itemize}
\item $T(A, S) < \infty$ if and only if $\widetilde T(A, S) < \infty$; and
\item if $T(A, S) < \infty$ and $\widetilde T(A, S) < \infty$, then
  $T(A, S) = |\Aut(A)| \cdot \widetilde T(A, S)$.
\end{itemize}
We shall primarily work with embedding Ramsey degrees. The proof of Theorem~\ref{hall.thm.main}
is the only place in the paper where structural Ramsey degrees are more convenient.

\paragraph{Edge-labelled graphs.}
Let $L$ be a countably infinite set of labels.
An $L$-edge-labelled graph
is a structure of the form $(V, E, \ell)$
where $V$ is the nonempty set of vertices of the graph,
$E \subseteq \binom V2$ is its set of edges, and
and $\ell : E \to L$ is the labelling function.
A complete $L$-edge-labelled graph
is an $L$-edge-labelled graph with $E = \binom V2$.
For complete $L$-edge-labelled graphs,
instead of $(V, \binom V2, \ell)$, we simply write $(V, \ell)$.
We shall usually identify a complete $L$-edge-labelled graphs with its vertex set $V$ and write
$\ell^V$ for its labelling function.

A function $f : V \to W$ is an \emph{embedding} of complete $L$-edge-labelled graphs if
$\ell^V(\{u, v\}) = \ell^W(\{f(u), f(v)\})$ for all $\{u, v\} \in \binom V2$.
For every at most countable set of labels $L$, the class of all finite complete $L$-edge-labelled graphs is a
\Fraisse\ class. Let $R_L$ denote its \Fraisse\ limit.

\begin{THM} (cf.~\cite[Theorem 1.1]{hubicka-omega-e-lab-hypergr})\label{hall.thm.hubicka-RL}
  Let $L$ be a countably infinite set of labels, and let $e$ be any complete $L$-edge-labeled graph on
  two vertices. Then $e$ does not have finite big Ramsey degree in $R_L$.
\end{THM}

\paragraph{Categories.}
The notion of Ramsey degrees generalizes to locally small categories straightforwardly.
The benefit of moving to this level of generality comes from the fact that we can then use
the machinery of category theory to transport the results about Ramsey degrees from one context (category) to
another.

Let $\CC$ be a locally small category. For $A, S \in \Ob(\CC)$,
a \emph{big Ramsey degree of $A$  in $S$} is the least positive integer $n \in \NN$, if such an integer exists,
with the property that for every $k \in \NN$ and every coloring
$\chi : \hom(A, S) \to k$ one can find a $w \in \hom(S, S)$ satisfying $|\chi(w \cdot \hom(A, S))| \le n$.
We then write $T_\CC(A, S) = n$. If no such $n \in \NN$ exists we write $T_\CC(A, S) = \infty$.

Consider a finite, acyclic, bipartite digraph with loops 
where all the arrows go from one class of vertices into the other
and the out-degree of all the vertices in the first class is~2 (modulo loops):
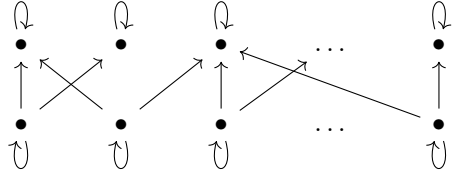
\begin{center}
  \begin{tikzcd}
    {\bullet} \arrow[loop above] & {\bullet} \arrow[loop above] & {\bullet} \arrow[loop above] & \ldots & {\bullet} \arrow[loop above] \\
    {\bullet} \arrow[loop below] \arrow[u] \arrow[ur] & {\bullet} \arrow[loop below] \arrow[ur] \arrow[ul] & {\bullet} \arrow[loop below] \arrow[u] \arrow[ur] & \ldots & {\bullet} \arrow[loop below] \arrow[u] \arrow[ull]
  \end{tikzcd}
\end{center}
Such a digraph can be thought of as a category where the loops represent the identity morphisms, and will be referred to as
a \emph{binary category}. (Note that all the compositions in a binary category
are trivial since no nonidentity morphisms are composable.)

An \emph{amalgamation problem} in a category $\CC$ is a diagram $F : \Delta \to \CC$ where $\Delta$ is a binary category,
$F$ takes the top row of $\Delta$ to the same object, and takes the bottom of $\Delta$ to the same object,
see Fig.~\ref{nrt.fig.2}. If $F$ takes the bottom row of $\Delta$ to an object $A$ and the top
row to an object $B$ then the diagram $F : \Delta \to \CC$ will be referred to as the \emph{$(A, B)$-diagram in $\CC$}.
An amalgamation problem $F : \Delta \to \CC$ \emph{has a solution in $\CC$} if $F$ has a compatible cocone
in $\CC$.

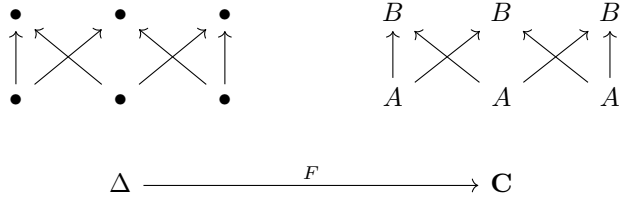
\begin{figure}
\centering
\begin{tikzcd}
  {\bullet} & {\bullet} & {\bullet}
  & & B & B & B
\\
  {\bullet} \arrow[u] \arrow[ur] & {\bullet} \arrow[ur] \arrow[ul] & {\bullet} \arrow[ul] \arrow[u]
  & & A \arrow[u] \arrow[ur] & A \arrow[ur] \arrow[ul] & A \arrow[ul] \arrow[u]
\\
  & \Delta \arrow[rrrr, "F"]  & & & & \CC  
\end{tikzcd}
\caption{An $(A,B)$-diagram in $\CC$ (of shape $\Delta$)}
\label{nrt.fig.2}
\end{figure}

\begin{THM}\label{bigrd.thm.1}\cite{masul-bigrd}
  Let $\BB$ and $\CC$ be locally finite categories whose every morphism is mono,
  and let $E : \BB \to \CC$ be a faithful functor.
  Let $B \in \Ob(\BB)$ be universal for $\BB$ and let $C \in \Ob(\CC)$ be universal for $G(\BB)$.
  Take any $A \in \Ob(\BB)$ and assume that for every $(A, B)$-diagram $F : \Delta \to \BB$ in $\BB$
  the following holds: if the amalgamation problem $EF : \Delta \to \CC$ has a solution in $\CC$ whose tip is $C$,
  then $F$ has a solution in~$\BB$. Then $T_\BB(A, B) \le T_\CC(E(A), C)$.
\end{THM}





\bigskip

Let $\CC$ be a category and $A, B, S \in \Ob(\CC)$. Then
$S$ \emph{is weakly homogeneous for $(A, B)$}, if there exist
$f \in \hom(A, B)$ and $g \in \hom(S, S)$ such that $g \cdot \hom(A, S) \subseteq \hom(B, S) \cdot f$.

\begin{THM}\label{rdbas.thm.mon-bigdeg}\cite{masul-rdbas}
  Let $\CC$ be a category whose morphisms are mono and let $A, B \in \Ob(\CC)$ be such that $A \to B$.
  Then $T(A, S) \le T(B, S)$ for every $S \in \Ob(\CC)$ which is weakly homogeneous for $(A, B)$.
\end{THM}



\section{Representing complete edge-labelled graphs by groups}
\label{hall.sec.ftr}

Let $L$ be a nonempty set of labels. To a complete $L$-edge-labelled graph $V$
we shall assign a group $\frakg(V)$ generated by its elements $\{[v] : v \in V\} \subseteq \frakg(V)$
so that the vertices $v \in V$ correspond to group elements $[v] \in \frakg(V)$,
while the information that the edge $\{u, v\}$ of $V$ is labelled by~$m \in L$
will be encoded by the request that $\langle[u], [v]\rangle_{\frakg(V)} \cong G_m$,
where $\{G_m : m \in L\}$ is a special family of finite groups.

\begin{DEF}\label{hall.def.lab-fam}
  Fix an integer $n \ge 2$.
  
  We say that a finite group $G$ \emph{has a pair of distinguished generators of order~$n$}
  if $G = \langle a, b\rangle$ for some $a \ne b$, $\ord(a) = \ord(b) = n$, and
  there are retractions $r^a : G \to \langle a \rangle$ and $r^b : G \to \langle b \rangle$
  such that $r^a(a) = a$, $r^a(b) = 1$, and $r^b(b) = b$, $r^b(a) = 1$.
  
  We say that a family $\calF$ of finite groups is a \emph{labelling family of order $n$} if
  there are infinitely many pairwise non-isomorphic groups in $\calF$, and each $G \in \calF$
  has a pair of distinguished generators of order~$n$.
\end{DEF}

\begin{EX}\label{hall.ex.dihedral}
  For $n \in \NN$ let $D_n$ denote the dihedral group on $2n$ elements:
  $$
    D_{n} = \langle a, b \mid a^2 = 1, b^2 = 1, (ab)^{n} = 1 \rangle.
  $$
  Note that in case $n = 2m$ there are retractions $r^a : D_{2m} \to \langle a\rangle$ and 
  $r^b : D_{2m} \to \langle b\rangle$ as in Definition~\ref{hall.def.lab-fam}, so $\calD = \{D_{2m} : m \in \NN\}$
  is a labelling family of order~2.
\end{EX}

\begin{THM}\label{hall.thm.FP}
  For a prime $p$, let $\calF_p$ be the class of all finite non-cyclic $p$-groups generated by two elements of order~$p$.
  Then $\calF_p$ is a labelling family of order~$p$.
\end{THM}
\begin{proof}
  See Appendix~\ref{hall.sec.A}.
\end{proof}

Fix an $n \ge 2$ and a labelling family $\calF = \{G_m : m \in \NN\}$ of order~$n$.
Without loss of generality we may assume that $G_i \not\cong G_j$ for all $i \ne j$.
Let $L \subseteq \NN$ be a nonempty set of labels.
For a complete $L$-edge-labelled graph $V$ we construct the group $\frakg(V)$ as follows.
Given $i \ne j \in V$, let $m_{ij} = \ell^V(\{i,j\})$.
We shall write $G^{\{i,j\}}$ for the group generated by $i$ and $j$ which is
isomorphic to $G_{m_{ij}} \in \calF$ under the isomorphism that takes $i$ and $j$ to the distingushed
generators of $G_{m_{ij}}$.

For a nonempty set $V$ let
$$
  I(V) = \{\{i, j\} : i, j \in V\}. 
$$
Note that $I(V)$ contains all singletons (in case $i = j$), and all unordered pairs of
elements of $V$. Therefore, whenever we take an arbitrary $\{i, j\} \in I(V)$ it may
happen that $i = j$. We shall find it convenient to denote by $G^{\{a\}}$ the cyclic group of order $n$ generated by~$a$.

Let
$$
  \Pi(V) = \prod\nolimits_{\{i,j\} \in I(V)} G^{\{i,j\}}.
$$
Just to fix terminology, each $x \in \Pi(V)$ is a function
$$
  x : I(V) \to \bigcup\nolimits_{\{i,j\} \in I(V)} G^{\{i,j\}}
$$
such that $x(\{i,j\}) \in G^{\{i,j\}}$ for all $\{i,j\} \in I(V)$.
As usual, the group operations are defined coordinatewise, e.g.:
$$
  (x\cdot y)(\{i,j\}) = x(\{i,j\}) \cdot y(\{i,j\}).
$$
For each $v \in V$ let $[v]$ denote the following element of $\Pi(V)$:
$$
  [v](\{i,j\}) = \begin{cases}
    v, & v \in \{i,j\}\\
    1, & \text{otherwise}.
  \end{cases}
$$
Finally, let $\frakg(V)$ be the subgroup of $\Pi(V)$
generated by $\{[v] : v \in V\}$.
Note that if $V$ is a finite complete edge-labelled graph then $\frakg(V)$ is also finite
(as a subgroup of a finite product of finite groups).

\begin{LEM}\label{hall.lem.orders}
  Let $L \subseteq \NN$ be a nonempty set of labels.
  Let $V$ be a complete $L$-edge-labelled graph, and
  let $\frakg(V)$ be the group constructed as above with the generating set
  $\{[v] : v \in V\}$ defined as above. Then, for all
  $u \ne v \in V$:
  
  $(a)$ $\ord_{\frakg(V)}([v]) = n$; and
  
  $(b)$ $\langle[u],[v]\rangle_{\frakg(V)} \cong G_{m_{uv}}$.
\end{LEM}
\begin{proof}
  $(a)$ If $[v](\{i,j\}) \ne 1$ then $[v](\{i,j\}) = i$ or $[v](\{i,j\}) = j$.
  Since $[v] \ne 1$, we get $\ord_{\frakg(V)}([v]) = \ord_{G_{m_{ij}}}(i) = \ord_{G_{m_{ij}}}(j) = n$.

  $(b)$ 
  Let show that $\langle[u],[v]\rangle_{\frakg(V)} \cong G^{\{u,v\}}$ since $G^{\{u,v\}}$ is a shorthand for $G_{m_{uv}}$.
  For each $\{u, v\} \in I(V)$, let us introduce several obvious homomorphisms:
  \begin{itemize}
  \item
    $e^{\{u,v\}}_v : G^{\{v\}} \to G^{\{u,v\}}$ is the embedding $v \mapsto v$;
  \item
    $r^{\{u,v\}}_v : G^{\{u,v\}} \to G^{\{v\}}$ is the retraction $u \mapsto 1$, $v \mapsto v$ whose existence is established by Definition~\ref{hall.def.lab-fam}.
  \end{itemize}
  Let
  $
    f = [f^{\{i,j\}} : \{i,j\} \in I(V)] : G^{\{u,v\}} \to \Pi(V)
  $
  be the cotupling of morphisms $f^{\{i,j\}} : G^{\{u,v\}} \to G^{\{i,j\}}$ defined as follows:
  \begin{itemize}
  \item $f^{\{u,v\}} = \id$;
  \item $f^{\{v\}} = r^{\{u,v\}}_v$;
  \item $f^{\{i\}} = 1$ if $i \notin \{u,v\}$;
  \item $f^{\{i,v\}} = e^{\{i,v\}}_v \circ r^{\{v\}}$ for $i \notin \{u, v\}$;
  \item $f^{\{i,j\}} = 1$ for $\{i, j\} \cap \{u, v\} = \0$.
  \end{itemize}
  Being a cotupling of homomorphisms, $f$ is a homomorphism. More over, $f$ is injective because $f^{\{u,v\}} = \id$.
  Clearly, $f(u) = [u]$ and $f(v) = [v]$, so $G^{\{u,v\}} \cong \im(f) \cong \langle[u],[v]\rangle_{\frakg(V)}$.
\end{proof}

\begin{THM}\label{hall.thm.emb-ext}
  Let $L \subseteq \NN$ be a nonempty set of labels, and
  let $V$ and $W$ be complete $L$-edge-labelled graphs.
  For every embedding $f : V \to W$ there is an embedding
  $\frakg(f) : \frakg(V) \to \frakg(W)$ such that
  $\frakg(f)([v]) = [f(v)]$ for all $v \in V$.
  
  This turns $\frakg$ into a functor from the category of complete $L$-edge-labelled graphs and
  edge-labelled graph embeddings into the category of groups and group embeddings
  which takes finite complete $L$-edge-labelled graphs to finite groups.
\end{THM}
\begin{proof}
  Fix an embedding $f : V \to W$.
  For each $\{u, v\} \in I(V)$, let us introduce several easy homomorphisms:
  \begin{itemize}
  \item
    $e^{\{u,v\}}_v : G^{\{v\}} \to G^{\{u,v\}}$ is the embedding $v \mapsto v$;
  \item
    $s_f^{\{u,v\}} : G^{\{u,v\}} \to G^{\{f(u),f(v)\}}$ is the obvious isomorphism $u \mapsto f(u), v \mapsto f(v)$
  \item
    $\pi^{\{u, v\}} : \Pi(V) \to G^{\{u, v\}} : x \mapsto x(\{u,v\})$ is the canonical projection; and
  \item
    $1^{\{u,v\}} : \Pi(V) \to G^{\{u,v\}}$ is the constant homomorphism that takes $\Pi(V)$ to $1 \in G^{\{u,v\}}$.
  \end{itemize}
  Note that these notions make sense even in case $u = v$. Note, also, that $s_f^{\{u,v\}}$
  is an isomorphism for $u \ne v$ because $f$ is an embedding, so ${m_{uv}} = \ell^V(\{u,v\}) =
  \ell^W(\{f(u),f(v)\}) = {m_{f(u)f(v)}}$. Let
  $$
    \hat{f} = [\hat{f}^{\{p,q\}} : \{p,q\} \in I(W)],
  $$
  be the cotupling of morphisms $\hat{f}^{\{p,q\}}$ defined as follows:
  \begin{itemize}
  \item
    $\hat{f}^{\{f(i),f(j)\}} = s^{\{i,j\}}_f \circ \pi^{\{i,j\}}$,
  \item
    $\hat{f}^{\{f(i),q\}} = e^{\{f(i),q\}}_{f(i)} \circ s^{\{i\}}_f \circ \pi^{\{i\}}$ if $q \notin f(V)$,
  \item
    $\hat{f}^{\{p,q\}} = 1^{\{p,q\}}$ if $p \notin f(V)$ and $q \notin f(V)$,
  \item
    $\hat{f}^{\{f(i)\}} = s^{\{i\}}_f \circ \pi^{\{i\}}$, and
  \item
    $\hat{f}^{\{p\}} = 1^{\{p\}}$ if $p \notin f(V)$.
  \end{itemize}
  Being a cotupling of homomorphisms, $\hat{f} : \Pi(V) \to \Pi(W)$ is a homomorphism.

  Let us show that $\hat{f}$ is injective. Take any $x, y \in \Pi(V)$ such that $x \ne y$.
  then there is a coordinate $\{i,j\} \in I(V)$ such that $x(\{i,j\}) \ne y(\{i,j\})$.
  Then $\hat{f}(x)$ and $\hat{f}(y)$ differ on the coordinate $\{f(i),f(j)\} \in I(W)$:
  \begin{multline*}
    \hat{f}(x)(\{f(i),f(j)\}) = \hat{f}^{\{f(i),f(j)\}}(x) =
    s^{\{i,j\}}_f \circ \pi^{\{i,j\}}(x) = s^{\{i,j\}}_f(x(\{i,j\})) \ne\\
    \ne s^{\{i,j\}}_f(y(\{i,j\})) = s^{\{i,j\}}_f \circ \pi^{\{i,j\}}(y) = \hat{f}^{\{f(i),f(j)\}}(y) = \hat{f}(y)(\{f(i),f(j)\})
  \end{multline*}
  because $s^{\{i,j\}}_f$ is an injective. So, $\hat{f}$ is an embedding.
  
  Next, let us show that $\hat{f}([v]) = [f(v)]$ for all $v \in V$ by showing that
  $\hat{f}^{\{p,q\}}([v]) = [f(v)](\{p,q\})$ for every $\{p,q\} \in I(W)$. The proof
  reduces to a straightforward verification:
  \begin{itemize}
  \item
    $\{p,q\} = \{f(v),f(j)\}$ for some $j \in V$:
    $\hat{f}^{\{f(v),f(j)\}}([v]) = s^{\{v,j\}}_f([v](\{v,j\})) = s^{\{v,j\}}_f(v) = f(v) = [f(v)]^{\{f(v),f(j)\}}$
    (this argument also applies in case $j = v$);
  \item
    $\{p,q\} = \{f(i),f(j)\}$ for some $v \notin \{i,j\} \in I(V)$:
    $\hat{f}^{\{f(i),f(j)\}}([v]) = s^{\{i,j\}}_f([v](\{i,j\})) = s^{\{i,j\}}_f(1) = 1 = [f(v)]^{\{f(i),f(j)\}}$,
    since $f(v) \notin \{f(i),f(j)\}$
    (this argument also applies in case $i = j$);
  \item
    $\{p,q\} = \{f(v),q\}$ for some $q \notin f(V)$:
    $\hat{f}^{\{f(v),q\}}([v]) = e^{\{f(v),q\}}_{f(v)} \circ s^{\{v\}}_f([v]^{\{v\}}) =
    e^{\{f(v),q\}}_{f(v)} \circ s^{\{v\}}_f(v) = e^{\{f(v),q\}}_{f(v)}(f(v)) = f(v) =
    [f(v)]^{\{f(v),q\}}$;
  \item
    $\{p,q\} = \{f(i),q\}$ for some $i \ne v$ and $q \notin f(V)$:
    $\hat{f}^{\{f(i),q\}}([v]) = e^{\{f(i),q\}}_{f(i)} \circ s^{\{i\}}_f ([v]^{\{i\}})=
    e^{\{f(i),q\}}_{f(i)} \circ s^{\{i\}}_f (1) = 1 = [f(v)]^{\{f(i),q\}}$ since $f(v) \notin
    \{f(i), q\}$;
  \item
    $\{p,q\} \cap f(V) = \0$:
    $\hat{f}^{\{p,q\}}([v]) = 1^{\{p,q\}}([v]) = 1 =[f(v)]^{\{p,q\}}$ (that this argument
    also applies in case $p = q$).
  \end{itemize}

  Therefore, $\hat{f}$ takes generators of $\frakg(V)$ to generators of $\frakg(W)$,
  so the restriction $\hat{f}|_{\frakg(V)}$ is an embedding $\frakg(V) \hookrightarrow
  \frakg(W)$.

  Define $\frakg(f) = \hat{f}|_{\frakg(V)}$.

  The proof that $\frakg$ is functorial is now easy: $\frakg(\id_V) = \id_{\frakg(V)}$
  since $\frakg(\id_V)$ is the identity on the generators of $\frakg(V)$. Finally,
  let $f : U \hookrightarrow V$ and $g : V \hookrightarrow W$ be embeddings
  between $L$-edge labelled graphs. To show that $\frakg(g \circ f) = \frakg(g) \circ \frakg(f)$
  it suffices to show that $\frakg(g \circ f)([u]) = \frakg(g) \circ \frakg(f)([u])$
  for every $u \in U$, but that is easy: $\frakg(g)(\frakg(f)([u])) = \frakg(g)([f(u)]) =
  [g(f(u))] = \frakg(g \circ f)([u])$.
\end{proof}

Note that in case $L = \{1, 2\}$ the category of complete $L$-edge-labelled graphs and embeddings
is isomorphic to the category of simple graphs and graph embeddings, so as an immediate
corollary of the above theorem we have:

\begin{COR}
  There is a functor from the category of simple graphs and graph embeddings to the
  category of groups and group embeddings which takes finite graphs to finite groups.
\end{COR}

\section{Big Ramsey degrees in Hall's universal group}
\label{hall.sec.brd}

Fix an integer $n \ge 2$ and a labelling family $\calF = \{G_m : m \in \NN\}$ of order $n$.
Again, we shall assume that $G_i \not\cong G_j$ for all $i \ne j$. Let
$$
  \Gamma = \{a \in H : \ord_H(a) = n\}.
$$
where $H$ denotes the Hall's universal group.
Let us define an $\NN$-edge-labelled graph on $\Gamma$ as follows.
For a pair of distinct vertices $a, b \in \Gamma$:
\begin{itemize}
\item
  if $\langle a, b\rangle_H \cong G_{m}$ put an edge between $a$ and $b$ and label the edge with $m$;
\item
  otherwise, $a$ and $b$ are non-adjacent in $\Gamma$.
\end{itemize}
Note that $\Gamma$ is a countably infinite, not necessarily complete $\NN$-edge-labelled graph.

\begin{LEM}
  $R_\NN \hookrightarrow \Gamma$.
\end{LEM}
\begin{proof}
  Enumerate vertices of $R_\NN$ as $v_1, v_2, \ldots$
  and let $V_k = R_\NN[v_1, \ldots, v_k]$ be the subgraph of $R_\NN$ induced by
  the set of vertices $\{v_1, \ldots, v_k\}$, $k \in \NN$.
  Each $V_k$ is a finite complete $\NN$-edge-labelled graph and the inclusion maps
  $\iota_k : V_k \to V_{k+1} : v_i \mapsto v_i$, $1 \le i \le k$, are embeddings.
  Clearly, $R_\NN = \bigcup_{k \in \NN} V_k$.
  
  Our aim is to build a copy of $R_\NN$ in $\Gamma$ via $H$ as follows.
  Fix an embedding $\phi_1 : \frakg(V_1) \to H$. Since $H$ is a \Fraisse\ limit,
  there exist embeddings $\phi_k : \frakg(V_k) \to H$, $k \ge 2$, such that the diagram below
  commutes:
  \begin{center}
    \begin{tikzcd}
      \frakg(V_1) \arrow[drrrr, bend right, "\phi_1"'] \arrow[r, "\frakg(\iota_1)"] & \frakg(V_2) \arrow[drrr, bend right=20, "\phi_2"'] \arrow[r, "\frakg(\iota_2)"] & \ldots & \frakg(V_k) \arrow[l, leftarrow, "\frakg(\iota_{k-1})"'] \arrow[r, "\frakg(\iota_k)"] \arrow[dr, bend right=15, "\phi_k"'] & \ldots \\
             &         &        &         & H
    \end{tikzcd}
  \end{center}
  Let $a_k = \phi_k([v_k]) \in H$, $k \in \NN$.
  By Lemma~\ref{hall.lem.orders} we have that $\ord_H(a_k) = \ord_{\frakg(V_k)}([v_k]) = n$,
  so $a_k \in \Gamma$. Moreover, for all $i < j$ we have that
  $\langle a_i, a_j\rangle_H \cong G_{m_{ij}}$.
  Therefore, the subgraph of $\Gamma$ induced by $\{a_k : k \in \NN\}$ is isomorphic to
  $R_\NN$ and the isomorphism is given by $v_i \mapsto a_i$, $i \in \NN$.
\end{proof}

\begin{LEM}\label{hall.lem.edge-inf}
  Let $e$ be any edge of\/ $\Gamma$. Then $T(e, \Gamma) = \infty$.
\end{LEM}
\begin{proof}
  We apply Theorem~\ref{bigrd.thm.1} to Theorem~\ref{hall.thm.hubicka-RL}.
  Let $\BB$ be the category whose objects are all finite
  and countably infinite induced subgraphs of $R_\NN$ together with embeddings, and let
  $\CC$ be the category whose objects are all finite
  and countably infinite induced subgraphs of $\Gamma$ together with embeddings.
  Because $R_\NN \hookrightarrow \Gamma$ we have that $\BB$ is a full subcategory of $\CC$,
  so in the context of Theorem~\ref{bigrd.thm.1} let $E : \BB \to \CC$ be the inclusion functor
  $E(A) = A$ and $E(f) = f$. This is clearly a faithful functor.
  
  Fix an edge $e$ of $\Gamma$. Let $F : \Delta \to \BB$ be an $(e, R_\NN)$-diagram and assume that
  $EF : \Delta \to \CC$ has a solution in $\CC$ whose tip is $\Gamma$. Let $\Gamma^*$ be
  the graph constructed from $\Gamma$ as follows: for every pair of non-adjacent vertices in $\Gamma$
  add an edge to $\Gamma^*$ between them labelled by~1. Then replacing $\Gamma$ with $\Gamma^*$ in the
  tip of the solution of the diagram $EF : \Delta \to \CC$ still works with the same embeddings.
  Since $\Gamma^*$ is a countable $\NN$-edge-labelled graph, it belongs to $\BB$, so this is the solution of
  $F$ in $\BB$. So, the Theorem~\ref{bigrd.thm.1} applies and $T_\BB(e, R_\NN) \le T_\CC(e, \Gamma)$.
  By Theorem~\ref{hall.thm.hubicka-RL} we have that $T_\BB(e, R_\NN) = \infty$. Therefore,
  $T_\CC(e, \Gamma) = \infty$.
\end{proof}

\begin{THM}\label{hall.thm.main}
  Let $n \ge 2$ be an integer and let $\calF$ be a labelling family of order $n$.
  Then $T(G, H) = \infty$ for every $G \in \calF$. Consequently,
  the Hall's universal group does not have finite big Ramsey degrees.
\end{THM}
\begin{proof}
  Let $\calF = \{G_m : m \in \NN\}$ be a labelling family of order $n \ge 2$
  such that $G_i \ne G_j$ for all $i \ne j$, and let $\Gamma$ be an $\NN$-edge-labelled graph
  constructed at the beginning of the section with respect to $\calF$.
  Fix an $m \in \NN$. Without loss of generality we may assume that $G_m \le H$.

  In this proof it is more convenient to work with structural Ramsey degrees. Recall that
  \begin{itemize}
  \item $T(G_m, H) = \infty$ if and only if $\widetilde T(G_m, H) = \infty$; and
  \item $T(e, \Gamma) = \infty$ if and only if $\widetilde T(e, \Gamma) = \infty$.
  \end{itemize}
  
  So, assume that $\widetilde T(G_m, H) = t \in \NN$ for some $m \in \NN$.
  Seeking a contradiction, fix an edge $e = \{u, v\}$ of $\Gamma$ labelled by $m$ and let
  $$
    d = \big|\big\{\{u,v\} \subseteq \Gamma \cap G_m : u \ne v \text{ and } \langle u, v \rangle_H = G_m\big\}\big|.
  $$
  We are going to show that $\widetilde T(e, \Gamma) \le d \cdot t$, contradicting Lemma~\ref{hall.lem.edge-inf}.
  
  Take any $k \in \NN$ and any coloring $\chi : \binom \Gamma e \to k$.
  Recall that
  \begin{align*}
    \binom \Gamma e
      &= \big\{\{u, v\} \subseteq \Gamma : u \ne v \text{ and } \langle u,v\rangle_H \cong G_m\big\}\\
      &= \bigcup_{G \in \binom{H}{G_m}} \big\{\{u, v\} \subseteq \Gamma : u \ne v \text{ and } \langle u,v\rangle_H = G\big\}.
  \end{align*}
  Define $\overline\chi : \binom{H}{G_m} \to \calP(k)$ as follows:
  $$
    \overline\chi(G) = \big\{\chi(\{u, v\}) : \{u, v\} \subseteq \Gamma \cap G, \ u \ne v \text{ and } \langle u, v \rangle_H = G\big\}.
  $$
  Since $\widetilde T(G_m, H) = t$, there is a $H' \le H$ such that $H' \cong H$ and there are
  $G^{(1)}, \ldots, G^{(t)} \in \binom{H'}{G_m}$ such that
  $$\textstyle
    \left\{\overline\chi(G) : G \in \binom{H'}{G_m}\right\} = \big\{\overline\chi(G^{(1)}), \ldots, \overline\chi(G^{(t)})\big\}
  $$
  Let $\Gamma'$ be the following induced subgraph of $\Gamma$:
  $$
    \Gamma' = \{a \in H' : \ord_{H}(a) = n\}.
  $$
  Is is clear that $\Gamma \cong \Gamma'$ because $H \cong H'$. Then:
  \begin{align*}
    \big\{\chi(e') &\textstyle: e' \in \binom{\Gamma'}{e}\big\}\\
    &= \big\{\chi(\{u, v\}) : \{u, v\} \subseteq \Gamma', u \ne v \text{ and } \langle u,v\rangle_{H} \cong G_m\big\}\\
    &= \bigcup_{G \in \binom{H'}{G_m}} \big\{\chi(\{u, v\}) : \{u, v\} \subseteq \Gamma \cap G, \ u \ne v \text{ and } \langle u, v \rangle_H = G\big\}\\
    &= \bigcup_{G \in \binom{H'}{G_m}} \overline\chi(G) = \bigcup_{i=1}^t \overline\chi(G^{(i)}).
  \end{align*}
  Hence,
  $$
    {\textstyle\Big|\big\{\chi(e') : e' \in \binom{\Gamma'}{e}\big\}\Big|} \le \sum_{i=1}^t |\overline\chi(G^{(i)})| \le d \cdot t.
  $$
  Thus, the assumption that $\widetilde T(G_m, H) = t \in \NN$ leads to conclusion that
  $\widetilde T(e, \Gamma) < \infty$, which contradicts Lemma~\ref{hall.lem.edge-inf}.
\end{proof}

\begin{COR}
  Let $\calF$ be a labelling family of order $n \ge 2$ and let $\overline\calF$ be the class of all
  finite groups having a subgroup isomorphic to a member of $\calF$.
  Then $T(\overline G, H) = \infty$ for every $\overline G \in \overline \calF$.
\end{COR}
\begin{proof}
  Take any $\overline G \in \overline \calF$ and any $G \in \calF$ such that $G \hookrightarrow \overline G$.
  Since $H$ is a \Fraisse\ limit, it is weakly homogeneous for $(G, \overline G)$. Theorem~\ref{rdbas.thm.mon-bigdeg}
  then yields $\infty = T(G, H) \le T(\overline G, H)$.
\end{proof}

\begin{COR}\label{hall.cor.main}
  $(a)$ $T(D_{2m}, H) = \infty$ for every $m \in \NN$.
  
  $(b)$ $T(S_n, H) = \infty$ for all $n \ge 3$, where $S_n$ denotes the symmetric group.
  
  $(c)$ If $G$ is a finite non-cyclic abelian group then $T(G, H) =  \infty$.
\end{COR}
\begin{proof}
  $(a)$ follows from the fact that dihedral groups $D_{2m}$ form a labelling family (Example~\ref{hall.ex.dihedral}).
  
  $(b)$ follows from $D_6 \cong S_3$, and $D_4 \cong \ZZ_2 \times \ZZ_2 \le S_n$ for $n \ge 4$.
  
  $(c)$ If $G$ is a finite non-cyclic abelian group, then there is a prime $p$ such that $\ZZ_p \times \ZZ_p \le G$. The claim
  now follows from $\ZZ_p \times \ZZ_p \in \calF_p$.
\end{proof}

\paragraph{Open problem 1.} Is there a finite group $G \ne 1$ with $T(G, H) < \infty$?
In particular, is it true that $T(\ZZ_n, H) < \infty$ for all $n \ge 2$, where $\ZZ_n$ denotes the cyclic group of order~$n$?

\bigskip

\begin{COR}
  Let $\calF_* = \bigcup_{p \text{ prime}} \calF_p$.
  Then $T(\overline G, H) = \infty$ for every $\overline G \in \overline{\calF_*}$.
\end{COR}

\paragraph{Open problem 2.} Resolve the status of big Ramsey degrees of groups which are not in $\overline{\calF_*}$.
These are exactly the groups such that every $p$-subgroup has a unique cyclic subgroup
(otherwise there are two elements of order $p$ so that the subgroup generated by them is a $p$-group that is not cyclic).
As for their $p$-subgroups we have the following:
if $G$ is a finite $p$-group such that it has a unique cyclic subgroup, then it is either cyclic if $p\neq 2$, or a generalized
quaternion group $Q_{2^k}$ for some $k$~\cite[Theorem 5.4.10]{Gorenstein}.
Therefore, all the groups which are not in $\overline{\calF_*}$ are solvable.

\appendix

\section{The existence of labelling families of any prime order}
\label{hall.sec.A}

The purpose of this section is to prove Theorem~\ref{hall.thm.FP}.
Let us start by collecting some facts about finite groups.

\begin{THM}\label{hall.thm.collect-1}
  $(a)$ \cite[Exercise 26, p.~199]{df}
  Let $G$ be a finite $p$-group and denote by $\Phi(G)$ its Frattini subgroup.
  Then $G/\Phi(G) \cong \ZZ_p^k$, where $k$ is the size of any minimal generating set of $G$.
  In particular, all minimal generating sets of $G$ have the same size.
  
  $(b)$ \cite{UrlNilp}
  A finitely generated nilpotent group whose generators have finite order is finite.
  
  $(c)$ \cite[Corollary 8.9]{nilSy}
  A finite group is nilpotent if and only if all its Sylow subgroups are normal.
  
  $(d)$ There is a unique Sylow $p$-subgroup for every $p$ in a finite nilpotent group.
\end{THM}
\begin{proof}
  $(d)$ Follows from $(c)$ as all Sylow $p$-subgroups are conjugate by the Sylow theorem.
\end{proof}

We say that a group $G$ is a \emph{resiudally finite $p$-group} if for every non-identity element $g \in G$
there is a homomorphism to a finite $p$-group that maps $g$ to a non-identity element.

\begin{THM} \cite{Magnus} \label{hall.thm.respg}
    If $G$ and $H$ are residually finite $p$-groups, then so is $G * H$.
    Consequently, $\ZZ_p * \ZZ_p$ is a residually finite $p$-group.
\end{THM}

\bigskip

We say a group $G$ is \emph{residually nilpotent} \cite{resnil1,resnil2}
if $\bigcap_{i \in \NN}\gamma_i(G)=1$. Here, $\gamma_i(G)=[\ldots[[G,G],G],\ldots,G]$ ($i$ times),
where $[\cdot,\cdot]$ denotes the subgroup generated by commutators.
Equivalently, a group is residually nilpotent if for every non-trivial element there
is a homomorphism to a nilpotent group sedning it to a non-trivial element.
Since every finite $p$-group is nilpotent, we have:

\begin{LEM}\label{hall.lem.res-nilp}
  Every residually finite $p$-group is a residually nilpotent group.
\end{LEM}

\begin{DEF}
 We say a finite group $G$ is an \emph{edge $p$-group} if it is a non-cyclic $p$-group generated by two elements of order $p$.
\end{DEF}

\begin{LEM}
    Let $G$ be an edge $p$-group generated by $g,h$ of order $p$. Then there are retractions $f_1,f_2:G\to G$ onto subgroups $\langle g\rangle,\langle h\rangle$ respectively.
\end{LEM}
\begin{proof}
    It suffices to show that there exists a retraction to $\langle g\rangle$.
    By Theorem~\ref{hall.thm.collect-1}~$(a)$ we have $G/\Phi(G)\cong \ZZ_p^k$,
    where $k$ is the size of the smallest  generating set of $G$, which is 2 since $G$ is 2-generated and non-cyclic.
    Let $\phi$ be the canonical projection. Vectors $\phi(g),\phi(h)\in \ZZ_p \times \ZZ_p$ are linearly independent over $\mathbb{F}_p$
    since they generate the Frattini quotient. Therefore there is some automorphism of the vector space
    (which is in particular a group automorphism) $\psi : \ZZ_p^2 \to \ZZ_p^2$ sending $\phi(g)$ to $(1,0)$ and $\phi(h)$ to $(0,1)$.
    Let $\pi_1$ be the first projection and $i : \ZZ_p \to \langle g\rangle$ be the isomorphism of cyclic groups.
    Then $f_1 =i\circ \pi_1\circ \psi\circ \phi$ is the desired retraction. To see this, it is enough to show
    $f_1(g)=g$ and $f_1(h)=1$. But this is easy:
    \begin{align*}
      f_1(g)&=i(\pi_1(\psi(\phi(g))))=i(\pi_1((1,0)))=i(1)=g,\\
      f_1(h)&=i(\pi_1(\psi(\phi(h))))=i(\pi_1(0,1))=i(0)=1.
    \end{align*}
    This finishes the proof.
\end{proof}

\begin{LEM}
   Up to isomorphism, there are infinitely many edge $p$-groups. 
\end{LEM}
\begin{proof}
    Let $G=\langle a,b\mid a^p=b^p=1\rangle$ and let $\gamma_k(G)$ be defined inductively as follows:
    \begin{itemize}
        \item $\gamma_1(G)=G$,
        \item $\gamma_{k+1}(G)=[\gamma(G),G]$.
    \end{itemize}
    Define $G_k=G/\gamma_k(G)$ for $k\geq 2$.
    This is well-defined since every $\gamma_k(G)$ is in fact a characteristic subgroup of $G$ (an easy induction).
    We show that $G_k$ are non-isomorphic edge $p$-groups. Note that, by construction, $G_k$ are nilpotent groups since the lower central
    series terminates in the trivial group. By Theorem~\ref{hall.thm.collect-1}~$(b)$, $G_k$ are all finite groups.
    Then by Theorem~\ref{hall.thm.collect-1}~$(d)$, we get that all elements of order $p$ lie in the same, unique Sylow $p$-subgroup,
    and therefore $a/\gamma_k(G)$ and $b/\gamma_k(G)$ lie in a $p$-subgroup of $G_k$, but they are generators so $G_k$ is a $p$-group.

    Now we show that elements $a,b$ in $G_k$ are of order exactly $p$, i.e. that $a/\gamma_k(G)\neq1\neq b/\gamma_k(G)$.
    It is easy to see that $\gamma_1(G)\geq\gamma_2(G)\geq \ldots \geq 1$, therefore by the Second Isomorphism Theorem
    (details in the equation bellow), there are quotient maps $G_1\twoheadleftarrow G_2\twoheadleftarrow G_3\twoheadleftarrow\ldots\twoheadleftarrow G$.
    Therefore, if images of $a,b$ in $G_2$ are non-trivial, they are non-trivial for all $k$. But this is true, since
    $G_2=\langle a,b\mid a^p = b^p = 1, ab=ba\rangle\cong \ZZ_p \times \ZZ_p$.

    Finally, to see that they are pairwise non-isomorphic, we show that $|G_k|<|G_{k+1}|$. Again, using the Second Isomorphism theorem,
    $$
      G_{k}=G/\gamma_{k}(G)\cong\frac{G/\gamma_{k+1}(G)}{\gamma_{k}(G)/\gamma_{k+1}(G)}=G_{k+1}/K,
    $$
    where $K = \gamma_{k}(G)/\gamma_{k+1}(G)$.
    Therefore we only need to show that $K$ is non-trivial, which is equivalent to
    $\gamma_k(G)\neq \gamma_{k+1}(G)$, i.e. $[\gamma_{k}(G),G]\neq \gamma_k(G)$.
    
    Assume, to the contrary, that $[\gamma_k(G),G]=\gamma_k(G)$. Then $\gamma_l(G)=\gamma_k(G)$ for all $l>k$ and
    $\bigcap_{i\in\mathbb{N}}\gamma_i(G)=\gamma_k(G)$. Since $G=\ZZ_p*\ZZ_p$ is a residually finite $p$-group (Theorem~\ref{hall.thm.respg}),
    it is residually nilpotent (Lemma~\ref{hall.lem.res-nilp}). Therefore, $\gamma_k(G) = \bigcap_{i\in\mathbb{N}}\gamma_i(G) = 1$, by definition
    of residual nilpotence.
    
    On the other hand, $\gamma_k(G)$ cannot be trivial.
    Namely, $G=\ZZ_p*\ZZ_p$, and thus all alternating words are non-zero. Actually, one can easily find a non-zero word in
    $\gamma_k(G)$, for example $[\ldots[[a,b],b]\ldots,b]$. Let us show inductively that
    the reduced word for $[\ldots[[a,b],b]\ldots,b]$ begins with $ab^x$ with $x\neq 0$ and ends
    with $a^yb^{-1}$ where $y\neq 0$. This in particular shows every such reduced word is non-empty.
    The base case $aba^{-1}b^{-1}$ is obvious. Assume it is true for $w_k$. We show this is true for
    \begin{multline*}
      w_{k+1}=[w_k,b]=w_kbw_k^{-1}b^{-1}=\\
      =(ab^{x}w'a^y b^{-1})b(ba^{-y}w^{'-1}b^{-x}a^{-1})b^{-1}=\\
      =ab^{x}w'a^{y}ba^{-y}w^{'-1}b^{-x}a^{-1}b^{-1}.
    \end{multline*}
    A word is reduced over a free product iff it is an alternating product of non-trivial elements, and the word above is reduced since $b^xw'a^y$
    is a subword of a reduced word and therefore reduced, and the part that we know letter-wise is already written in the alternating form.
    
    This contradiction shows that $K$ is non-trivial, leaading to the conclusion that $|G_k|<|G_{k+1}|$ for all~$k$.
\end{proof}

\begin{THM}[Theorem~\ref{hall.thm.FP}]
    The set $\mathcal{F}_p$ of all edge $p$-groups is a labelling family of order~$p$.
\end{THM}
\begin{proof}
    We only need to show that the family is infinite and that retractions exist,
    since we already know the groups are non-cyclic and every member of the family is 2-generated by elements of order~$p$.
    But this was done in the above two lemmas.
\end{proof}

\section*{Acknowledgements}

The initial version of this paper was written during our stay at the XXXI Midsummer Combinatorial Workshop
in Prague, Czech Republic, organized in honor of Jaroslav Ne\v set\v ril's 80th birthday.
We would like to thank the organizers, and in particular Jan Jubi\v cka for many
fruitful discussions and great hospitality.

The first author gratefully acknowledges the financial support of the Ministry of Science,
Technological Development and Innovation of the Republic of Serbia
(Grants No. 451-03-33/2026-03/ 200125 \& 451-03-34/2026-03/ 200125).

The second author gratefully acknowledges
funding by the Deutsche Forschungsgemeinschaft (DFG, German Research
Foundation) under Germany's Excellence Strategy -- The Berlin Mathematics
Research Center MATH+ (EXC-2046/1, EXC-2046/2, project ID: 390685689).

\end{document}